\documentclass{amsart}
\calclayout
 \usepackage{amssymb,amsmath,amsfonts,epsfig,latexsym,tikz}
 \usepackage[alphabetic]{amsrefs}
 \usepackage{tikz-cd}
\usepackage{hyperref}
\usepackage{enumerate}
\usepackage{mathtools}
\usepackage{verbatim}
\usepackage{cleveref}
\usepackage[shortlabels]{enumitem}
\usepackage{subcaption}

\usetikzlibrary{positioning}
\usetikzlibrary{matrix}
\usetikzlibrary{decorations}
\usetikzlibrary{decorations.pathreplacing, decorations.pathmorphing, angles,quotes}
 
 \newtheorem{theorem}{Theorem}[section]
\newtheorem{proposition}[theorem]{Proposition}
\newtheorem{lemma}[theorem]{Lemma}
\newtheorem{corollary}[theorem]{Corollary}
\newtheorem{conjecture}[theorem]{Conjecture}

\theoremstyle{definition}
\newtheorem{definition}[theorem]{Definition}
\newtheorem{remark}[theorem]{Remark}
\newtheorem{example}[theorem]{Example}

\begin{document}

\title{Rank functions and invariants of delta-matroids}          
    
\author{Matt Larson}
\date{\today}
\address{Institute for Advanced Study and Princeton University}
\email{mattlarson@princeton.edu}
\begin{abstract}
In this note, we give a rank function axiomatization for delta-matroids and study the corresponding rank generating function. We relate an evaluation of the rank generating function to the number of independent sets of the delta-matroid, and we prove a log-concavity result for that evaluation using the theory of Lorentzian polynomials. 
\end{abstract}
 
\maketitle

\vspace{-20 pt}

\section{Introduction}
Let $[n, \overline{n}]$ denote the set $\{1, \dotsc, n, \overline{1}, \dotsc, \overline{n}\}$, equipped with the obvious involution $\overline{(\cdot)}$. 
Let $\operatorname{AdS}_n$ be the set of \emph{admissible subsets} of $[n, \overline{n}]$, i.e., subsets $S$ that contain at most one of $i$ and $\overline{i}$ for each $i \in [n]$. These are also called \emph{partial transversals}. Set $e_{\overline{i}} := - e_i \in \mathbb{R}^n$, and for each $S \in \operatorname{AdS}_n$, set $e_S = \sum_{a \in S} e_a$. 

\begin{definition}\label{def:delta}
A \emph{delta-matroid} $D$ is a non-empty collection $\mathcal{F} \subset \operatorname{AdS}_n$ of admissible sets of size $n$, called the \emph{feasible sets} of $D$, such that the polytope
$$P(D) := \operatorname{Conv}\{e_B: B \in \mathcal{F}\}$$
has all edges parallel to $e_i$ or $e_i \pm e_j$, for some $i, j$. We say that $D$ is \emph{even} if all edges of $P(D)$ are parallel to $e_i \pm e_j$. 
\end{definition}

Delta-matroids were introduced in \cite{Bou87} by replacing the usual basis exchange axiom for matroids with one involving symmetric difference. In \cite{Bou87}, the above definition is called a \emph{symmetric matroid}. They were defined independently  in \cite{ChandrasekaranKabadi,DressHavel}. For the equivalence of the definition of delta-matroids in those works with the one given above, and for general properties of delta-matroids, see \cite[Chapter 4]{BGW}.

A delta-matroid is even if and only if all sets in $\{B \cap [n] : B \in \mathcal{F}\}$ have the same parity. Even delta-matroids enjoy nicer properties than arbitrary delta-matroids. For instance, they satisfy a version of the symmetric exchange axiom \cite{WenzelExchange}.

There are many constructions of delta-matroids in the literature. Two of the most fundamental come from matroids: given a matroid $M$ on $[n]$, we can construct a delta-matroid on $[n, \overline{n}]$ whose feasible sets are the sets of the form $B \cup \overline{B^c}$, for $B$ a basis of $M$. We can also construct a delta-matroid whose feasible sets are the sets of the form $I \cup \overline{I^c}$, for $I$ independent in $M$. Additionally, there are delta-matroids corresponding to graphs \cite{Duchamp}, graphs embedded in surfaces \cites{CMNR19a,CMNR19b}, and points of a maximal orthogonal or symplectic Grassmannian. Delta-matroids arising from points of a maximal orthogonal or symplectic Grassmannian are called \emph{realizable}. 
See \cite[Section 6.2]{EFLS} for a discussion of delta-matroids associated to points of a maximal orthogonal Grassmannian. 

Given $S, T \in \operatorname{AdS}_n$, we define $S \sqcup T = \{a \in S \cup T \colon \overline{a} \not \in S \cup T\}$. A function $f \colon \operatorname{AdS}_n \to \mathbb{R}$ is called \emph{bisubmodular} if, for all $S, T \in \operatorname{AdS}_n$, 
$$f(S) + f(T) \ge f(S \cap T) + f(S \sqcup T).$$
There is a large literature on bisubmodular functions, beginning with \cite{DunstanWelsh}. They have been studied both from an optimization perspective \cites{FujishigeIwata,FujishigeParametric} and from a polytopal perspective \cites{FujishigePatkar, FujishigeBisubmodular}. Additionally, bisubmodular functions are closely related to jump systems \cite{BouchetCunningham}.

For a delta-matroid $D$, define a function $g_D \colon \operatorname{AdS}_n \to \mathbb{Z}$ by 
$$g_D(S) = \max_{B \in \mathcal{F}} (|S \cap B| - |\overline{S} \cap B|).$$
We call $g_D$ the \emph{rank function} of $D$. Note that $g_D$ may take negative values. 
The collection of feasible subsets of $D$ is exactly $\{S : g_D(S) = n\}$, so $D$ can be recovered from $g_D$. 

\begin{theorem}\label{thm:rank}
A function $g \colon \operatorname{AdS}_n \to \mathbb{Z}$ is the rank function of a delta-matroid if and only if
\begin{enumerate}
\item $g(\emptyset) = 0$ (normalization),
\item $|g(S)| \le 1$ if $|S| = 1$ (boundedness),
\item $g(S) + g(T) \ge g(S \cap T) + g(S \sqcup T)$ (bisubmodularity), and
\item $g(S) \equiv |S| \pmod 2$ (parity).
\end{enumerate}
Furthermore, $D$ is even if and only if
$$g_D(S) = \frac{g_D(S \cup i) + g_D(S \cup \overline{i})}{2} \text{ whenever }|S| = n-1 \text{ and }\{i, \overline{i}\} \cap S = \emptyset.$$
\end{theorem}

The function $g_D$, as well as the observation that it is bisubmodular, has appeared before in the literature \cites{BouchetRep,ChandrasekaranKabadi}. For example, in \cite[Theorem 4.1]{BouchetRep} it is shown that, if $D$ is represented by a point of the maximal symplectic Grassmannian, then $g_D$ can be computed in terms of the rank of a certain matrix. 
It was known that delta-matroids admit a description in terms of certain bisubmodular functions. Theorem~\ref{thm:rank} answers a special case of \cite[Question 9.4]{ACEP20}.

In \cite{BouShelter,MultimatroidsII}, Bouchet gave a rank-function axiomatization of delta-matroids in the more general setting of multimatroids. His rank function differs from ours --- in Section~\ref{ssec:bouchet}, we discuss the relationship between his results and Theorem~\ref{thm:rank}. An axiomatization of delta-matroids in terms of Bouchet's rank function was given in an unpublished paper of Allys. One can deduce Theorem~\ref{thm:rank} from this result, see Corollary~\ref{cor:hdcharacterization}.

Basic operations on delta-matroids --- like products, deletion, contraction, and projection --- can be simply expressed in terms of rank functions. See Section~\ref{ssec:construction}.

One of the most important invariants of a matroid $M$ of rank $r$ on $[n]$ is its \emph{Whitney rank generating function}. If $\operatorname{rk}_M$ is the rank function of $M$, then the rank generating function is defined as
$$R_M(u, v) := \sum_{A \subset [n]}u^{r - \operatorname{rk}_M(A)}v^{|A| - \operatorname{rk}_M(A)}.$$
The more commonly used normalization is the \emph{Tutte polynomial}, which is $R_M(u-1, v-1)$. 
The characterization of delta-matroids in terms of rank functions allows us to consider an analogously-defined invariant.

\begin{definition}\label{def:U}
Let $D$ be a delta-matroid on $[n, \overline{n}]$. Then we define
$$U_D(u, v) = \sum_{S \in \operatorname{AdS}_n} u^{n - |S|} v^{\frac{|S| - g_D(S)}{2}}.$$
\end{definition}

Note that the bisubmodularity of $g_D$ implies that the restriction of $g_D$ to the subsets of any fixed $S \in \operatorname{AdS}_n$ is submodular. The boundedness of $g_D$ then implies that $|g_D(S)| \le |S|$. Because of the parity requirement, $|S| - g_D(S)$ is divisible by $2$. Therefore $U_D(u, v)$ is indeed a polynomial. The normalization $U_D(u-1, v-1)$ is more analogous to the Tutte polynomial, but it can have negative coefficients. However, the polynomial $U_D(u, v-1)$ has non-negative coefficients (as follows, e.g., from Theorem~\ref{thm:activity}).

The $U$-polynomial of a delta-matroid was introduced by Eur, Fink, Spink, and the author in \cite[Definition 1.4]{EFLS} in terms of a Tutte polynomial-like recursion; see Proposition~\ref{prop:recursive} for a proof that Definition~\ref{def:U} agrees with the recursive definition considered there. 
The specialization $U_D(0, v)$ is the \emph{interlace polynomial} of $D$, which was introduced in \cite{InterlaceBollobas} for graphs and in \cite{BrijderInterlace} for general delta-matroids. See \cite{MorseInterlace} for a survey on the properties of the interlace polynomial. 

Various Tutte polynomial-like invariants of delta-matroids have been considered in the literature, such as the Bollob\'{a}s--Riordan polynomial and its specializations \cite{Bollobas}. In \cite{HopfTutte}, a detailed analysis of delta-matroid polynomials which satisfy a deletion-contraction formula is carried out. Set $\sigma_D(A) = \frac{|A|}{2} + \frac{g_D(A) + g_D(\bar{A})}{4}$ for $A \subset [n]$. Then in \cite{HopfTutte}, the polynomial
$$\sum_{A \subset [n]} (x-1)^{\sigma_D([n]) - \sigma_D(A)}(y-1)^{|A| - \sigma_D(A)}$$
is shown to be, in an appropriate sense, the universal invariant of delta-matroids which satisfies a deletion-contraction formula. This polynomial is a specialization of the Bollob\'{a}s--Riordan polynomial. 
In \cite{IrreducibleTutte}, it is shown that this polynomial has several nice combinatorial properties. This polynomial does not specialize to $U_{D}(u, v)$.

\begin{example}\cite[Example 5.5 and 5.6]{EFLS}
Let $M$ be a matroid of rank $r$ on $[n]$, and let $S = S^+ \cup \overline{S^-} \in \operatorname{AdS}_n$ be an admissible set with $S^+, S^- \subset [n]$. Set $V = \{i \in [n] : S \cap \{i, \bar{i}\} = \emptyset\}$. Previously, we gave two examples of delta-matroids constructed from $M$. We now discuss their $U$-polynomials. 
\begin{enumerate}
\item Let $D$ be the delta-matroid arising from the independent sets of $M$. Then $g_D(S) = |S| + 2\operatorname{rk}_{{M}}(S^+) - 2|S^{+}|$, and
$$U_D(u, v) = (u + 1)^{n-r} R_M\left( u + 3, \frac{ 2u + v + 2}{u + 1} \right).$$
\item Let $D$ be the delta-matroid arising from the bases of $M$. Then $g_D(S) = |S| - 2r +  2 \operatorname{rk}_M(S^+ \cup V) - 2|S^+| + 2 \operatorname{rk}_M(S^+)$, and 
$$U_D(u, v) = \sum_{T \subset S \subset [n]} u^{|S \setminus T|} v^{r - \operatorname{rk}_M(S) + |T| - \operatorname{rk}_M(T)}.$$
\end{enumerate}
\end{example}

We study the $U$-polynomial as a delta-matroid analogue of the rank generating function of a matroid. For a matroid $M$, the evaluation $R_M(u, 0)$ is essentially the $f$-vector of the independence complex of the matroid, i.e., it counts the number of independent sets of $M$ of a given size. The coefficients of the Tutte polynomial $R_M(u-1, v-1)$ can be interpreted as counting bases of $M$ according to their internal and external activities, certain statistics that depend on an ordering of the ground set. See \cite{BackmanTutte}. This shows that $R_M(u, -1)$, the (unsigned) characteristic polynomial of $M$, is essentially the $f$-vector of the broken circuit complex of $M$. 

A set $S \in \operatorname{AdS}_n$ is \emph{independent} if it is contained in a feasible set of a delta-matroid $D$. In \cite{BouShelter}, Bouchet gave an axiomatization of delta-matroids in terms of their independent sets. 
The independent sets form a simplicial complex, called the \emph{independence complex} of $D$. 
We relate $U_D(u, 0)$ to the $f$-vector of the independence complex of $D$ (Proposition~\ref{prop:indep}), which gives linear inequalities between the coefficients of $U_D(u, 0)$. 
We give a combinatorial interpretation of the coefficients of $U_D(u, v-1)$ as counting the number of independent sets of $D$ of a given size according to a delta-matroid version of activity (Theorem~\ref{thm:activity}) which was introduce by Morse \cite{MorseActivity}. 
This shows that $U_D(u, -1)$ is essentially the $f$-vector of a certain simplicial complex associated to $D$.

Following a tradition in matroid theory (see, e.g., \cite{Mason}), and inspired by the ultra log-concavity of $R_M(u, 0)$ \cites{ALGV,BH}, we make three log-concavity conjectures for $U_D(u, 0)$. These conjectures state the sequence of the number of independent sets of a delta-matroid of a given size satisfies log-concavity properties. 
\begin{conjecture}\label{conj:logconc}
Let $D$ be a delta-matroid on $[n, \overline{n}]$, and let $U_D(u, 0) = a_n + a_{n-1}u + \dotsb + a_0 u^n$. Then, for any $k \in \{1, \dotsc, n-1\}$,
\begin{enumerate}
\item\label{conj1} $a_k^2 \ge \frac{n-k+1}{n-k} a_{k+1}a_{k-1}$,
\item\label{conj2} $a_k^2 \ge \frac{2n - k + 1}{2n-k}\frac{k+1}{k} a_{k+1}a_{k-1}$, and
\item $a_k^2 \ge \frac{n - k + 1}{n-k}\frac{k+1}{k} a_{k+1}a_{k-1}$.
\end{enumerate}
\end{conjecture}

Conjecture~\ref{conj:logconc}(1) follows from \cite[Conjecture 1.5]{EFLS}, and it is proven in \cite[Theorem B]{EFLS} when $D$ has an \emph{enveloping matroid} (see Definition~\ref{def:env}). This is a technical condition which is satisfied by many commonly occurring delta-matroids, including all realizable delta-matroids and delta-matroids arising from matroids (although not all delta-matroids, see \cite[Section 4]{BouShelter} and \cite[Example 6.11]{EFLS}). The proof uses algebro-geometric methods. Here we prove a special case of Conjecture~\ref{conj:logconc}(2). Note that Conjecture~\ref{conj:logconc}(1) and Conjecture~\ref{conj:logconc}(2) are incomparable, and both are implied by Conjecture~\ref{conj:logconc}(3). Equality is attained in Conjecture~\ref{conj:logconc}(3) if $D$ has a unique feasible set or if all admissible sets are independent.

\begin{theorem}\label{theorem:logconc}
Let $D$ be a delta-matroid on $[n, \overline{n}]$ which has an enveloping matroid. Let $U_D(u, 0) = a_n + a_{n-1}u + \dotsb + a_0 u^n$. Then, for any $k \in \{1, \dotsc, n-1\}$, $a_k^2 \ge \frac{2n - k + 1}{2n-k}\frac{k+1}{k} a_{k+1}a_{k-1}$, i.e., Conjecture~\ref{conj:logconc}(2) holds. 
\end{theorem}
Our argument uses the theory of Lorentzian polynomials \cite{BH}. We strengthen Theorem~\ref{theorem:logconc} by proving that a generating function for the independent sets of $D$ is Lorentzian (Theorem~\ref{thm:lor}), which implies the desired log-concavity statement. 
We deduce that this generating function is Lorentzian from the fact that the Potts model partition function of an enveloping matroid is Lorentzian \cite[Theorem 4.10]{BH}.

When $D$ is the delta-matroid arising from the independent sets of a matroid, Conjecture~\ref{conj:logconc}(3) follows from the ultra log-concavity of the number of independent sets of that matroid \cites{ALGV,BH}. When $D$ is the delta-matroid arising from the bases of a matroid $M$ on $[n]$, which has an enveloping matroid by \cite[Proposition 6.10]{EFLS}, Theorem~\ref{theorem:logconc} gives a new log-concavity result. If we set
$$a_k = |\{T \subset S \subset [n] : T \text{ independent in $M$ and $S$ spanning in $M$, } |S \setminus T| = n - k\}|,$$
then Theorem~\ref{theorem:logconc} gives that $a_k^2 \ge \frac{2n - k + 1}{2n-k}\frac{k+1}{k} a_{k+1}a_{k-1}$ for $k \in \{1, \dotsc, n-1\}$. 
\\ 
\noindent
\textbf{Acknowledgements:} We thank Nima Anari, Christopher Eur, Alex Fink, Satoru Fujishige, Steven Noble, and Hunter Spink for enlightening conversations, and we thank Christopher Eur, Steven Noble, Shiyue Li, and the referees for helpful comments on a previous version of this paper.  The author is supported by an NDSEG fellowship.

\section{Rank functions of delta-matroids}\label{sec:rank}

The proof of Theorem~\ref{thm:rank} goes by way of a polytopal description of normalized bisubmodular functions, which we now recall. To a function $f \colon \operatorname{AdS}_n \to \mathbb{R}$ with $f(\emptyset) = 0$, we associate the polytope 
$$P(f) = \{x : \langle e_S, x \rangle \le f(S) \text{ for all non-empty }S \in \operatorname{AdS}_n\}.$$
By \cite[Theorem 4.5]{BouchetCunningham} (or \cite[Theorem 5.2]{ACEP20}), $P(f)$ has all edges parallel to $e_i$ or $e_i \pm e_j$ if and only if $f$ is bisubmodular. In this case, $P(f)$ is a lattice polytope if and only if $f$ is integer-valued. 
For a normalized (i.e., $f(\emptyset) = 0$) bisubmodular function $f$, we can recover $f$ from $P(f)$ via the formula
$$f(S) = \max_{x \in P(f)}\langle e_S, x \rangle.$$
Under this dictionary, the bisubmodular function corresponding to the dilate $k P(f)$ is $kf$, and the bisubmodular function corresponding to the Minkowski sum $P(f) + P(g)$ is $f + g$. 

\begin{proof}[Proof of Theorem~\ref{thm:rank}]
By the polyhedral description of normalized bisubmodular functions, for each delta-matroid $D$ there is a unique normalized bisubmodular function $g$ such that $P(D) = P(g)$. We show that the conditions on a normalized bisubmodular function $g$ for $P(g)$ to have all vertices in $\{-1, 1\}^n$ are exactly those given in Theorem~\ref{thm:rank}, namely that $|g(S)| \le 1$ when $|S| =1$ and $g(S) \equiv |S| \pmod 2$. 

The polytope $P(g)$ has all vertices in $\{\pm 1\}^n$ if and only if $\frac{1}{2}(P(g) + (1, \dotsc, 1))$ is a lattice polytope which is contained in $[0, 1]^n$. 
The normalized bisubmodular function $h$ corresponding to the point $(1, \dotsc, 1)$ takes value $h(S) = |S^+| - |S^{-}|$ on an admissible set of the form $S = S^+ \cup \overline{S^-}$, with $S^+, S^- \subset [n]$.
The polytope $\frac{1}{2}(P(g) + (1, \dotsc, 1))$ is $P(f)$, where $f$ is the normalized bisubmodular function defined by $f := \frac{1}{2}(g + h)$. We note that $P(f)$ is a lattice polytope which is contained in $[0, 1]^n$ if and only if 
\begin{enumerate}
\item $f(i) \in \{0, 1\}$ and $f(\overline{i}) \in \{-1, 0\}$, and
\item $f$ is integer-valued. 
\end{enumerate}
A normalized bisubmodular function $f$ satisfies these conditions if and only if $g$ satisfies the conditions of Theorem~\ref{thm:rank}, giving the characterization of rank functions of delta-matroids.

By \cite[Example 5.2.3]{ACEP20}, the polytope $P(g_D) = P(D)$ has all edges parallel to $e_i \pm e_j$ if and only if $g_D$ satisfies the condition $$g_D(S) = \frac{g_D(S \cup i) + g_D(S \cup \overline{i})}{2} \text{ whenever }|S| = n-1 \text{ and }\{i, \overline{i}\} \cap S = \emptyset.$$
This gives the characterization of even delta-matroids.
\end{proof}

\subsection{Compatibility with delta-matroid operations}\label{ssec:construction}

In this section, we consider several operations on delta-matroids, and we show that the rank function behaves in a simple way under these operations. First we consider minor operations on delta-matroids --- contraction, deletion, and projection. 

\begin{definition}
Let $D$ be a delta-matroid on $[n, \overline{n}]$ with feasible sets $\mathcal{F}$, and let $i \in [n]$. We say that $i$ is a \emph{loop} of $D$ if no feasible set contains $i$, and we say that $i$ is a \emph{coloop} if every feasible set contains $i$. 
\begin{enumerate}
\item If $i$ is not a loop of $D$, then the \emph{contraction} $D/i$ is the delta-matroid with feasible sets $B \setminus i$, for $B \in \mathcal{F}$ containing $i$. 
\item If $i$ is not a coloop of $D$, then the \emph{deletion} $D \setminus i$ is the delta-matroid with feasible sets $B \setminus \overline{i}$, for $B \in \mathcal{F}$ containing $\overline{i}$. 
\item The \emph{projection} $D(i)$ is the delta-matroid with feasible sets $B \setminus \{i, \overline{i}\}$ for $B \in \mathcal{F}$. 
\item If $i$ is a loop or coloop, then set $D/i = D \setminus i = D(i)$. 
\end{enumerate}
\end{definition}
For $A \subset [n]$, we define $D/A, D\setminus A,$ and $D(A)$ to be the delta-matroids on $[n, \bar{n}] \setminus (A \cup \bar{A})$ obtained by successively contracting, deleting, or projecting away from all elements of $A$. Contractions, deletions, and projections at disjoint sets commute with each other, so this is well defined. If $A$ and $B$ are disjoint subsets of $[n]$, then $D/A \setminus B$ is the delta-matroid obtained by contracting $A$ and then deleting $B$, which is the same as first deleting $B$ and then contracting $A$. 

First we describe the rank function of projections. The formula is analogous to the formula for the rank function of a matroid deletion. 

\begin{proposition}\label{prop:projection}
Let $D$ be a delta-matroid on $[n, \overline{n}]$, and let $A \subset [n]$. For each $S \in \operatorname{AdS}_n$ disjoint from $A \cup \overline{A}$, $g_{D(A)}(S) = g_D(S)$.
\end{proposition}

\begin{proof}
As $S$ is disjoint from $A \cup \overline{A}$, $|B \cap S| - |B \cap \overline{S}|$ depends only on $B \setminus (A \cup \overline{A})$. The feasible sets of $D(A)$ are given by $B \setminus (A \cup \overline{A})$ for $B$ a feasible set of $D$. 
\end{proof}

The rank functions of the contractions and  deletions are described by the following result. The formula is analogous to the formula for the rank function of a matroid contraction. 

\begin{proposition}\label{prop:minors}
Let $D$ be a delta-matroid on $[n, \overline{n}]$. 
Let $A, B \subset [n]$ be disjoint subsets, and let $S \in \operatorname{AdS}_n$ be disjoint from $A \cup B \cup \overline{A} \cup \overline{B}$. Then $g_{D/A \setminus B}(S) = g_D(S \cup A \cup \overline{B}) - g_D(A \cup \overline{B})$. 
\end{proposition}

Before proving this, we will need the following property of delta-matroids. It follows, for instance, from the greedy algorithm description of delta-matroids in \cite{BouchetCunningham}. 

\begin{proposition}\label{prop:greedy}
Let $D$ be a delta-matroid on $[n, \overline{n}]$, and let $S \subset T \in \operatorname{AdS}_n$. Let $\mathcal{F}_S$ be the collection of feasible sets $B$ of $D$ that maximize $|S \cap B|$, i.e., have $|S \cap B| = \max_{B' \in \mathcal{F}} |S \cap B'|$. Then 
$$\max_{B \in \mathcal{F}_S} |T \cap B| = \max_{B \in \mathcal{F}} |T \cap B|.$$ 
\end{proposition}

First we consider the case when we delete or contract a single element. 
\begin{lemma}\label{lem:minors}
Let $D$ be a delta-matroid on $[n, \overline{n}]$, and let $i \in [n]$. Then
\begin{enumerate}
\item If $i$ is not a loop, then $g_{D/ i}(S) = g_{D}(S \cup i) - 1$,
\item If $i$ is not a coloop, then $g_{D \setminus i}(S) = g_D(S \cup \overline{i}) - 1$. 
\end{enumerate}
\end{lemma}

\begin{proof}
We do the case of contraction; the case of deletion is identical. Assume that $i$ is not a loop, and let $\mathcal{F}_i$ denote the set of feasible sets of $D$ which contain $i$. Note that $\mathcal{F}_i$ is non-empty, so it is the collection of feasible sets $B$ of $D$ which maximize $|\{i\} \cap B|$. For any $S \in \operatorname{AdS}_n$ with $S \cap \{i, \overline{i}\} = \emptyset$, by Proposition~\ref{prop:greedy} we have that
$$\max_{B \in \mathcal{F}} |(S \cup i) \cap B| = \max_{B \in \mathcal{F}_i} |(S \cup i) \cap B|.$$
For any $B$, $|(S \cup i) \cap B| - |\overline{(S \cup i)} \cap B| = 2|(S \cup i) \cap B| - |S \cup i|$, so we see that
$$\max_{B \in \mathcal{F}} (|(S \cup i) \cap B|  - |\overline{(S \cup i)} \cap B|) = \max_{B \in \mathcal{F}_i} (|(S \cup i) \cap B|  - |\overline{(S \cup i)} \cap B|).$$
The left-hand side is equal to $g_D(S \cup i)$, and the right-hand side is equal to $g_{D/i}(S) + 1$. 
\end{proof}

\begin{proof}[Proof of Proposition~\ref{prop:minors}]
First note that $g_D(i) = 1$ if $i$ is not a loop and is $-1$ if $i$ is a loop, and similarly $g_D(\overline{i}) = 1$ if $i$ is not a coloop and is $-1$ is $i$ is a coloop. So Lemma~\ref{lem:minors} implies the result holds when $|A \cup B| = 1$. 

We induct on the size of $A \cup B$. We consider the case of adding an element $i \in [n]$ to $A$; the case of adding it to $B$ is identical. We compute:
\begin{equation*}\begin{split}
g_{D/(A \cup i) \setminus B}(S) &= g_{D/A \setminus B}(S \cup i) - g_{D/A \setminus B}(i) \\  
&= g_{D}(S \cup A \cup \overline{B} \cup i) - g_D(A \cup \overline{B}) - (g_{D}(A \cup \overline{B} \cup i) - g_D(A \cup \overline{B})) \\ 
&= g_D(S \cup (A \cup i) \cup \overline{B}) - g_D((A \cup i) \cup \overline{B}). \qedhere
\end{split}\end{equation*}
\end{proof}

For two non-negative integers $n_1, n_2$, identify the disjoint union of $[n_1]$ and $[n_2]$ with $[n_1 + n_2]$. 
Given two delta-matroids $D_1, D_2$ on $[n_1]$ and $[n_2]$, let $D_1 \times D_2$ be the delta-matroid on $[n_1 + n_2]$ whose feasible sets are $B_1 \cup B_2$, for $B_j$ a feasible set of $D_j$. Then we have the following description of the rank function of $D_1 \times D_2$.

\begin{proposition}\label{prop:product}
Let $D_1, D_2$ be delta-matroids on $[n_1]$ and $[n_2]$, and let $S = S_1 \cup S_2$ be an admissible subset of $[n_1 + n_2, \overline{n_1 + n_2}]$, with $S_1 \subset [n_1, \overline{n}_1]$ and $S_2 \subset [n_2, \overline{n}_2]$. Then $g_{D_1 \times D_2}(S) = g_{D_1}(S_1) + g_{D_2}(S_2)$. 
\end{proposition}

\begin{proof}
Let $B_1$ be a feasible set of $D_1$ with $g_{D_1}(S_1) = |S_1 \cap B_1| - |\overline{S_1} \cap B_1|$, and let $B_2$ be a feasible set of $D_2$ with $g_{D_2}(S_2) = |S_2 \cap B_2| - |\overline{S_2} \cap B_2|$. Then $B_1 \cup B_2$ maximizes $B \mapsto |S \cap B| - |\overline{S} \cap B|$, and so $g_{D_1 \times D_2}(S) = |S_1 \cap B_1| - |\overline{S_1} \cap B_1| + |S_2 \cap B_2| - |\overline{S_2} \cap B_2| = g_{D_1}(S_1) + g_{D_2}(S_2)$. 
\end{proof}

We now study how the rank function behaves under the operation of \emph{twisting}. Let $W$ be the \emph{signed permutation group}, the subgroup of the symmetric group on $[n, \overline{n}]$ which preserves $\operatorname{AdS}_n$. In other words, $W$ consists of permutations $w$ such that $w(\overline{i}) = \overline{w(i)}$. As delta-matroids are collections of admissible sets, $W$ acts on the set of delta-matroids on $[n, \bar{n}]$. This action is usually called twisting in the delta-matroid literature. 
\begin{proposition}\label{prop:twist}
Let $D$ be a delta-matroid on $[n, \overline{n}]$, and let $w \in W$. Then $g_{w \cdot D}(S) = g_{D}(w^{-1} \cdot S)$. 
\end{proposition}
\begin{proof}
Note that, for $B$ a feasible set of $D$, $|S \cap (w \cdot B)| - |\overline{S} \cap (w \cdot B)| = |(w^{-1} \cdot S) \cap B| - |\overline{(w^{-1} \cdot S)} \cap B|$, which implies the result. 
\end{proof}

Let $S \in \operatorname{AdS}_n$ be an admissible set of size $n$. For any delta-matroid $D$ on $[n, \overline{n}]$, let $r$ be the maximal value of $|S \cap B|$. Then $\{S \cap B \colon B \in \mathcal{F}, |S \cap B| = r\}$ is the set of bases of a matroid on $S$. When $S = [n]$, this is sometimes called the upper matroid of $D$. We describe the rank function of this matroid in terms of the rank function of $D$. 

\begin{proposition}\label{prop:maxmatroid}
Let $S \in \operatorname{AdS}_n$ be an admissible set of size $n$, and let $D$ be a delta-matroid on $[n, \overline{n}]$ with $r = \max_{B \in \mathcal{F}} |S \cap B|$. The matroid $M$ on $S$ whose bases are $\{S \cap B \colon B \in \mathcal{F}, |S \cap B| = r\}$ has rank function
$$\operatorname{rk}_M(T) = \frac{g_D(T) + |T|}{2}.$$
\end{proposition}

\begin{proof}
Let $\mathcal{F}_{S}$ be the collection of feasible sets $B$ with $|S \cap B| = r$. Then we have that
$$\operatorname{rk}_M(T) = \max_{B \in \mathcal{F}_{S}} |T \cap B| \le \max_{B \in \mathcal{F}} |T \cap B| = \frac{g_D(T) + |T|}{2}.$$
On the other hand, by Proposition~\ref{prop:greedy} there is a feasible set $B$ which maximizes $|T \cap B|$ and has $|S \cap B| = r$, so we have equality. 
\end{proof}

\subsection{An alternative normalization}\label{ssec:bouchet}

The results of the previous section, particularly Proposition~\ref{prop:maxmatroid}, suggest that an alternative normalization of the rank function of a delta-matroid has nice properties. Set
$$h_D(S) := \frac{g_D(S) + |S|}{2}.$$
The function $h_D(S)$ is integer-valued and bisubmodular, and the polytope it defines is $P(h_D) = \frac{1}{2}(P(D) + \square)$, where $\square = [-1, 1]^n$ is the cube and the sum is Minkowski sum.  This is because the bisubmodular function corresponding to $\square$ is $S \mapsto |S|$. Note that the function $h_D$ is non-negative and increasing, in the sense that if $S \subset T \in \operatorname{AdS}_n$, then $h_D(S) \le h_D(T)$. 
Theorem~\ref{thm:rank} implies the following characterization of the functions arising as $h_D$ for some delta-matroid $D$. In \cite[Theorem 2.16]{MultimatroidsII}, this characterization of the functions $h_D$ is stated with a reference to an unpublished paper of Allys. 
\begin{corollary}\label{cor:hdcharacterization}
A function $h \colon \operatorname{AdS}_n \to \mathbb{Z}$ is equal to $h_D$ for some delta-matroid $D$ if and only if
\begin{enumerate}
\item $h(\emptyset) = 0$,
\item $h(S) \in \{0, 1\}$ if $|S| = 1$,
\item $h(S) + h(T) \ge h(S \cap T) + h(S \sqcup T) + |S \cap \overline{T}|$. 
\end{enumerate}
\end{corollary}
Indeed, these are exactly the conditions we need for $g(S) := 2h(S) - |S|$ to satisfy the conditions in Theorem~\ref{thm:rank}. Note that one can also deduce Theorem~\ref{thm:rank} from Corollary~\ref{cor:hdcharacterization}.

The function $h_D$ was studied by Bouchet in \cite{BouShelter,MultimatroidsII} in the more general setting of multimatroids. The following alternative characterization of the functions $h_D$ follows from \cite[Proposition 4.2]{BouShelter}:
\begin{proposition}\label{prop:bouchetrank}
A function $h \colon \operatorname{AdS}_n \to \mathbb{Z}$ is equal to $h_D$ for some delta-matroid $D$ if and only if
\begin{enumerate}
\item $h(\emptyset) = 0$,
\item $h(S) \le h(S \cup a) \le h(S) + 1$ if $S \cup a$ is admissible,
\item $h(S) + h(T) \ge h(S \cap T) + h(S \cup T)$ if $S \cup T$ is admissible, and 
\item $h(S \cup i) + h(S \cup \bar{i}) \ge 2h(S) + 1$ if $S \cap \{i, \bar{i}\} = \emptyset$.
\end{enumerate}
\end{proposition}



\section{The $U$-polynomial}

We now study the $U$-polynomial of delta-matroids. We prove the following recursion for $U_D(u, v)$, which was the original definition of the $U$-polynomial in \cite[Definition 1.4]{EFLS}. 

\begin{proposition}\label{prop:recursive}
If $n = 0$, the $U_D(u, v) = 1$. For any $i \in [n]$, the $U$-polynomial satisfies
\begin{equation*}
U_D(u, v) = \begin{cases} U_{D/i}(u,v) + U_{D \setminus i}(u, v) + uU_{D(i)}(u, v), &\text{ $i$ is neither a loop nor a coloop} \\ (u + v + 1) \cdot U_{D \setminus i}(u, v), & \text{ $i$ is a loop or a coloop}.\end{cases} 
\end{equation*}
\end{proposition}

First we study the behavior of the $U$-polynomial under products. 

\begin{lemma}\label{lem:product}
Let $D_1, D_2$ be delta-matroids on $[n_1, \overline{n}_1]$ and $[n_2, \overline{n}_2]$. Then $U_{D_1 \times D_2}(u, v) = U_{D_1}(u, v) U_{D_2}(u, v)$. 
\end{lemma}
\begin{proof}
We compute:
\begin{equation*}\begin{split}
U_{D_1}(u, v) U_{D_2}(u, v) & = \left( \sum_{S_1 \in \operatorname{AdS}_{n_1}} u^{n_1 - |S_1|} v^{\frac{|S_1| - g_{D_1}(S_1)}{2}} \right)\left( \sum_{S_2 \in \operatorname{AdS}_{n_2}} u^{n_2 - |S_2|} v^{\frac{|S_2| - g_{D_2}(S_2)}{2}} \right) \\ 
&= \sum_{(S_1, S_2)} u^{n_1 + n_2 - |S_1| - |S_2|} v^{\frac{|S_1| + |S_2| - g_{D_1}(S_1) - g_{D_2}(S_2)}{2}} \\ 
&= \sum_{(S_1, S_2)} u^{n_1 + n_2 - |S_1| - |S_2|} v^{\frac{|S_1| + |S_2| - g_{D_1 \times D_2}(S_1 \cup S_2)}{2}} \\ 
&= U_{D_1 \times D_2}(u, v),
\end{split} \end{equation*}
where the third equality is Proposition~\ref{prop:product}.
\end{proof}

\begin{proof}[Proof of Proposition~\ref{prop:recursive}]
If $n = 0$, then the only admissible subset of $[n, \overline{n}]$ is the empty set, and $g_D(\emptyset) = 0$, so $U_D(u, v) = 1$. Now choose some $i \in [n]$. 

First suppose that $i$ is neither a loop nor a coloop. The admissible subsets of $[n, \overline{n}]$ are partitioned into sets containing $i$, sets containing $\overline{i}$, and sets containing neither $i$ nor $\overline{i}$. If $S$ contains $i$, then $u^{n - |S|}v^{\frac{|S| - g_D(S)}{2}} = u^{n - 1 - |S \setminus i|} v^{\frac{|S\setminus i| - g_{D/i}(S \setminus i)}{2}}$. If $S$ contains $\overline{i}$, then $u^{n - |S|}v^{\frac{|S| - g_D(S)}{2}} = u^{n - 1 - |S \setminus i|} v^{\frac{|S\setminus \overline{i}| - g_{D \setminus i}(S \setminus \overline{i})}{2}}$. If $S$ contains neither $i$ not $\overline{i}$, then $u^{n - |S|}v^{\frac{|S| - g_D(S)}{2}} = u \cdot u^{n - 1 - |S|} v^{\frac{|S| - g_{D(i)}(S)}{2}}$. Adding these up implies the recursion in this case. 

If $i$ is a loop or a coloop, then $D$ is the product of $D \setminus i$ with a delta-matroid on $1$ element with $1$ feasible set. We observe that $U$-polynomial of a delta-matroid on $1$ element with $1$ feasible set is $u + v + 1$, and so Lemma~\ref{lem:product} implies the recursion in this case. 
\end{proof}

\subsection{The independence complex of a delta-matroid}

In this section, we introduce the independence complex of a delta-matroid and use it to study the $U$-polynomial. 

\begin{definition}
We say that $S \in \operatorname{AdS}_n$ is \emph{independent} in $D$ if $g_D(S) = |S|$, or, equivalently, if $S$ is contained in a feasible subset of $D$. The \emph{independence complex} of $D$ is the simplicial complex on $[n, \overline{n}]$ whose facets are given by the feasible sets of $D$. 
\end{definition}

Let $S \in \operatorname{AdS}_n$, and let $T = \{i \in [n] : S \cap \{i, \overline{i}\} = \emptyset\}$. Note $S$ is independent if and only if $S$ is a feasible set of $D(T)$. 
\\ \noindent
The following result is immediate from the definition of $U_D(u, 0)$. 

\begin{proposition}\label{prop:indep}
Let $f_i(D)$ be the number of $i$-dimensional faces of the independence complex of $D$. Then $U_D(u, 0) = f_{n-1}(D) + f_{n-2}(D) u + \dotsb + f_{-1}(D) u^n$.
\end{proposition}

Note that the $f$-vector of a pure simplicial complex, like the independence complex of a delta-matroid, is a \emph{pure O-sequence}. Then \cite{Hibi89} gives the following inequalities. 

\begin{corollary}\label{cor:pureO}
Let $U_D(u, 0) = a_n + a_{n-1}u + \dotsb a_0u^n$. Then $(a_0, \dotsc, a_n)$ is the $f$-vector of a pure simplicial complex. In particular, $a_i \le a_{n - i}$ for $i \le n/2$ and $a_0 \le a_1 \le \dotsb \le a_{\lfloor \frac{n+1}{2} \rfloor}$. 
\end{corollary}

Proposition~\ref{prop:indep} is a delta-matroid analogue of the fact that, for a matroid $M$, the coefficients of $R_M(u, 0)$, when written backwards, are the face numbers of the independence complex of $M$. The independence complex of a matroid is shellable \cite{BjornerShellable}, which is reflected in the fact that $R_M(u-1, 0)$ has non-negative coefficients. The independence complex of a delta-matroid is not in general shellable or Cohen--Macaulay, and $U_D(u-1, 0)$ can have negative coefficients. 

Recall that $\square = [-1, 1]^n$ is the cube. 
The map $S \mapsto e_S$ induces a bijection between $\operatorname{AdS}_n$ and lattice points of $\square.$
We use this to give a polytopal description of the independent sets of $D$, which will be useful in the sequel. 
\begin{proposition}\label{prop:latticedescrip}
The map $S \mapsto e_S$ induces a bijection between independent sets of $D$ and lattice points in $\frac{1}{2}(P(D) + \square)$.
\end{proposition}

\begin{proof}
If $S$ is independent in $D$, then there is $T \in \operatorname{AdS}_n$ such that $S \cup T \in \mathcal{F}$. Then $e_S = \frac{1}{2}(e_{S \cup T} + e_{S \cup \overline{T}})$, so $e_S$ lies in $\frac{1}{2}(P(D) + \square)$. 

The correspondence between normalized bisubmodular functions and polytopes gives that 
$$\frac{1}{2}(P(D) + \square) = \left \{x : \langle e_S, x \rangle \le \frac{g_D(S) + |S|}{2} \right \}.$$
If $S$ is not independent, then $e_S$ violates the inequality $\langle e_S, e_S \rangle \le \frac{g_D(S) + |S|}{2}$, so $e_S$ does not lie in $\frac{1}{2}(P(D) + \square)$. 
\end{proof}

\subsection{The activity expansion of the $U$-polynomial}\label{ssec:activity}

We now discuss an expansion of $U_D(u, v-1)$ in terms of a statistic associated to each independent set of a delta-matroid $D$, similar to the expansion of the Tutte polynomial of a matroid in terms of basis activities. We rely heavily on the work of Morse \cite{MorseActivity}, who gave such an expansion for the interlace polynomial $U_D(0, v-1)$. Throughout we fix the ordering $1 < 2 < \dotsb < n$ on $[n]$. For $S \in \operatorname{AdS}_n$, let $\underline{S} \subset [n]$ denote the unsigned version of $S$, i.e., the image of $S$ under the quotient of $[n, \overline{n}]$ by the involution $\overline{(\cdot)}$. 

\begin{definition}
Let $B$ be a feasible set in a delta-matroid $D$. We say that $i \in [n]$ is \emph{$B$-orientable} if the symmetric difference $B \Delta \{i, \bar{i}\}$ is not a feasible set of $D$. We say that $i$ is \emph{$B$-active} if $i$ is $B$-orientable and there is no $j < i$ such that $B \Delta \{i, j, \bar{i}, \bar{j}\}$ is a feasible set of $D$. For an independent set $I$ of $D$, we say that $i \in \underline{I}$ is \emph{$I$-active} if $i$ is $I$-active in the projection $D([n] \setminus \underline{I})$. Let $a(I)$ denote the number of $i \in \underline{I}$ which are $I$-active. 
\end{definition}

\begin{theorem}\label{thm:activity}
Let $D$ be a delta-matroid on $[n, \bar{n}]$. Then
$$U_D(u, v-1) = \sum_{I \text{ independent in }D} u^{n - |I|} v^{a(I)}.$$
\end{theorem}

\begin{proof}
By \cite[Corollary 5.3]{MorseActivity}, this holds after we evaluate at $u=0$ for any delta-matroid $D$. By \cite[Proposition 5.2]{EFLS}, we have that
$$U_D(u, v-1) = \sum_{S \subset [n]} u^{n - |S|}U_{D([n] \setminus S)}(0, v-1).$$
The result follows because each independent set $I$ is a feasible set of exactly one projection of $D$. 
\end{proof}

Theorem~\ref{thm:activity} implies that the coefficient of  $u^{n-i}$ in $U_D(u, -1)$ counts the number of independent sets of size $i$ with $a(I) = 0$. This is analogous to how the coefficient of $u^{r - i}$ in $R_M(u, -1)$ counts the number of independent sets of external activity zero in a matroid $M$, which form the faces of dimension $i - 1$ in the broken circuit complex of $M$  \cite{BackmanTutte}. This interpretation in terms of a simplicial complex generalizes to delta-matroids. 

\begin{proposition}\label{prop:complex}
The independent sets $I$ of $D$ with $a(I) = 0$ form a simplicial complex on $[n, \bar{n}]$. 
\end{proposition}

\begin{proof}
It suffices to check that if $i$ is not $B$-active for some feasible set $B$ of $D$ and $S \subset [n] \setminus i$, then $i$ is not active for $B \setminus (S \cup \bar{S})$. Because $i$ is not $B$-active, either $B \Delta \{i, \bar{i}\}$ is feasible (which remains true after we project away from $S$), or there is $j < i$ such that $B \Delta \{i, j, \bar{i}, \bar{j}\}$ is feasible. If $j \not \in S$, then this remains true after we project away from $S$. If $j \in S$, then $i$ is not $B \setminus (S \cup \bar{S})$-orientable. 
\end{proof}

This complex can be complicated; for instance, its dimension is not easy to predict. The following example shows that the complex defined above need not be pure, so we cannot use it to deduce that $U_D(u, -1)$ is pure O-sequence as in Corollary~\ref{cor:pureO}. 

\begin{example}
Let $D$ be the delta-matroid on $[3, \bar{3}]$ with feasible sets $\{1, \bar{2}, \bar{3}\}, \{\bar{1}, 2, \bar{3}\}$, and $\{\bar{1}, \bar{2}, 3\}$. Every element of $[3, \bar{3}]$ has no active elements, the sets $\{\bar{1}, 2\}, \{\bar{1}, \bar{2}\}, \{\bar{2}, 3\}, \{\bar{2}, \bar{3}\}, \{\bar{1}, 3\}$, and $\{\bar{1}, \bar{3}\}$ are the independent sets of size $2$ with no active elements, and every feasible set has an active element. The complex defined in Proposition~\ref{prop:complex} has $f$-vector $(1, 6, 6)$, so $U_D(u, -1) = 6u + 6u^2 + u^3$. This complex is not pure because $1$ is not contained in any facet. 
\end{example}

\subsection{Enveloping matroids}
We now recall the definition of an enveloping matroid of a delta-matroid, which was introduced for algebro-geometric reasons in \cite[Section 6]{EFLS}. A closely related notion was considered in \cite{BouShelter}, see Remark~\ref{rem:sheltering}.

For $S \subseteq [n, \bar{n}]$, let $u_S$ denote the corresponding indicator vector in $\mathbb{R}^{[n, \bar{n}]}$. 
For a matroid $M$ on $[n, \bar{n}]$, let $P(M) = \operatorname{Conv}\{u_B : B \text{ basis of }M\}$, and let $IP(M) = \operatorname{Conv}\{u_S : S \text{ independent in M}\}$. 

\begin{definition}\label{def:env}
Let $\operatorname{env} \colon \mathbb{R}^{[n, \overline{n}]} \to \mathbb{R}^n$ be the map given by $(x_1, \dotsc, x_n, x_{\overline{1}}, \dotsc, x_{\overline{n}}) \mapsto (x_1 - x_{\overline{1}}, \dotsc, x_n - x_{\overline{n}})$. 
Let $D$ be a delta-matroid on $[n, \overline{n}]$, and let $M$ be a matroid on $[n, \overline{n}]$. We say that $M$ is an \emph{enveloping matroid} for $D$ if $\operatorname{env}(P(M)) = P(D)$. 
\end{definition}

Note that enveloping matroids necessarily have rank $n$. In \cite[Section 6.3]{EFLS}, it is shown that many different types of delta-matroids have enveloping matroids, such as realizable delta-matroids, delta-matroids arising from the independent sets or bases of a matroid, and delta-matroids associated to graphs or embedded graphs. We will need the following property of enveloping matroids. 

\begin{proposition}\label{prop:envelopingindep}
Let $M$ be an enveloping matroid for a delta-matroid $D$ on $[n, \overline{n}]$. Let $S \in \operatorname{AdS}_n$ be an admissible set. Then $S$ is independent in $M$ if and only if it is independent in $D$. 
\end{proposition}

\begin{proof}
If $S \in \operatorname{AdS}_n$, then $\operatorname{env}(u_S) = e_S$, and $S$ is the only admissible set with this property. Furthermore, if $S \in \operatorname{AdS}_n$ has size $n$, then $u_S$ is the only indicator vector of a subset of $[n, \bar{n}]$ of size $n$ which is a preimage of $e_S$ under $\operatorname{env}$. 
Because $\operatorname{env}(P(M)) = P(D)$, we see that if $B$ is a feasible set of $D$, then $B$ is a basis for $M$. This implies that the independent sets in $D$ are independent in $M$. 

By \cite[Lemma 7.6]{EFLS}, $\operatorname{env}(IP(M)) = \frac{1}{2}(P(D) + \square)$. If $S$ is admissible and independent in $M$, then $\operatorname{env}(u_S) = e_S \in \frac{1}{2}(P(D) + \square)$, so by Proposition~\ref{prop:latticedescrip}, $S$ is independent in $D$. 
\end{proof}

\begin{remark}\label{rem:sheltering}
A matroid $M$ on $[n, \bar{n}]$ is a \emph{sheltering matroid} for a delta-matroid $D$ if every independent set of $D$ is independent in $M$. Equivalently, $M$ is sheltering if the restriction of the rank function of $M$ to $\operatorname{AdS}_n$ is $h_D$. The proof of Proposition~\ref{prop:envelopingindep} shows that if $M$ is an enveloping matroid for $D$, then it is also a sheltering matroid. The converse is false, see \cite[Remark 6.7]{EFLS}. 
\end{remark}

\subsection{Lorentzian polynomials}

For a multi-index $\textbf{m} = (m_0, m_1, \dotsc)$, let $w^{\textbf{m}} = w_0^{m_0}w_1^{m_1} \dotsb$. A homogeneous polynomial $f(w_0, w_1, \dotsc)$ of degree $d$ with real coefficients is said to be \emph{strictly Lorentzian} if all its coefficients are positive, and the quadratic form obtained by taking $d-2$ partial derivatives is nondegenerate with exactly one positive eigenvalue. We say that $f$ is \emph{Lorentzian} if it is a coefficient-wise limit of strictly Lorentzian polynomials. Lorentzian polynomials enjoy strong log-concavity properties, and the class of Lorentzian polynomials is preserved under many natural operations. 

The following lemma is a special case of \cite[Proposition 3.3]{DualRoss}. Alternatively, it can be deduced from the proof of \cite[Corollary 3.5]{BH}. We thank Nima Anari for discussing this lemma with us. 

\begin{lemma}\label{lem:multiaffine}
For a polynomial $f(w_0, w_1, \dotsc) = \sum_{\textbf{m}} c_{\textbf{m}} w^{\textbf{m}}$, let
$$\overline{f}(w_0, w_1, \dotsc) = \sum_{\textbf{m}: m_i \le 1 \text{ for }i \not= 0} c_{\textbf{m}} w^{\textbf{m}}.$$
If $f$ is Lorentzian, then $\overline{f}$ is Lorentzian. 
\end{lemma}

For $S \in \operatorname{AdS}_n$, recall that $\underline{S} \subset [n]$ denotes the unsigned version of $S$. For a set $T$, let $w^T = \prod_{a \in T} w_a$. We now state a strengthening of Theorem~\ref{theorem:logconc}. 
\begin{theorem}\label{thm:lor}
Let $D$ be a delta-matroid on $[n, \overline{n}]$ which has an enveloping matroid. Then the polynomial
$$\sum_{S \text{ independent in }D} w_0^{2n - |S|} w^{\underline{S}} \in \mathbb{R}[w_0, w_1, \dotsc, w_n]$$
is Lorentzian. 
\end{theorem}
\begin{remark}
In \cite[Theorem 8.1]{EFLS}, it is proven that if $D$ has an enveloping matroid, then the polynomial
$$\sum_{S \text{ independent in }D} \frac{w_0^{|S|}}{|S|!} w^{[n] \setminus \underline{S}} \in \mathbb{R}[w_0, w_1, \dotsc, w_n]$$
is Lorentzian.
\end{remark}

\begin{proof}[Proof of Theorem~\ref{theorem:logconc}]
By \cite[Theorem 2.10]{BH}, the specialization 
$$\sum_{S \text{ independent in }D} w_0^{2n - |S|} y^{|S|} = \sum_{i=0}^{n} f_{i-1}(D) w_0^{2n - i} y^i$$ is Lorentzian. By \cite[Example 2.26]{BH}, the coefficients of a Lorentzian polynomial in two variables of degree $2n$ are log-concave after dividing the coefficient of $w_0^{2n - i}y^i$ by $\binom{2n}{i}$, which implies the result.  
\end{proof}

\begin{proof}[Proof of Theorem~\ref{thm:lor}]
Let $M$ be an enveloping matroid of $D$. By \cite[Proof of Theorem 4.14]{BH}, the polynomial
$$\sum_{S \text{ independent in }M} w_0^{2n - |S|} w^S \in \mathbb{R}[w_0, w_1, \dotsc, w_n, w_{\overline{1}}, \dotsc, w_{\overline{n}}]$$
is Lorentzian. Setting $w_{\overline{i}} = w_i$, by \cite[Theorem 2.10]{BH} the polynomial
$$\sum_{S \text{ independent in }M} w_0^{2n - |S|} w^{S \cap [n]} w^{\overline{S \cap [\overline{n}]}} \in \mathbb{R}[w_0, w_1, \dotsc, w_n]$$
is Lorentzian. A term $w_0^{2n - |S|} w^{S \cap [n]} w^{\overline{S \cap [\overline{n}]}}$ has degree at most $1$ in each of the variables $w_1, \dotsc, w_n$ if and only if $S$ is admissible, in which case it is equal to $w^{\underline{S}}$. Therefore, by Lemma~\ref{lem:multiaffine}, the polynomial 
$$\sum_{S \in \operatorname{AdS}_n \text{ independent in }M} w_0^{2n - |S|} w^{\underline{S}} \in \mathbb{R}[w_0, w_1, \dotsc, w_n]$$
is Lorentzian. By Proposition~\ref{prop:envelopingindep}, this polynomial is equal to the polynomial in Theorem~\ref{thm:lor}.
\end{proof}

\begin{remark}
Let $(U, \Omega, r)$ be a multimatroid \cite{BouShelter}, i.e., $U$ is a finite set, $\Omega$ is a partition of $U$, and $r$ is a function on partial transversals of $\Omega$ satisfying certain conditions. An \emph{independent set} is a partial transversal $S$ of $\Omega$ with $r(S) = |S|$. A multimatroid is called \emph{shelterable} if $r$ can be extended to the rank function of a matroid on $U$. Then the argument used to prove Theorem~\ref{theorem:logconc} shows that, if $a_k$ is the number of independent sets of a shelterable multimatroid of size $k$, then
$$a_k^2 \ge \frac{|U| - k + 1}{|U| - k} \frac{k+1}{k} a_{k+1}a_{k-1}.$$
In particular, the proof of Theorem~\ref{theorem:logconc} only requires $D$ to be shelterable. 
\end{remark}

\bibliographystyle{alpha}
\bibliography{matroid.bib}

\end{document}